\documentclass{amsart}
\usepackage{etoolbox}
\usepackage[hypertexnames=false]{hyperref}
\usepackage{microtype}
\usepackage{amssymb,dsfont,amsthm}
\usepackage{breqn}
\usepackage[nobreak]{cite}
\setkeys{breqn}{compact}
\usepackage[utf8]{inputenc}
\usepackage[T1]{fontenc}
\usepackage{lmodern}
\usepackage{enumitem}
\usepackage[capitalize,nameinlink]{cleveref}
\usepackage[english,polish]{babel}
\crefname{empty}{}{}
\newlist{alist}{enumerate}{1}
\setlist[alist]{label=(\alph*),itemsep=.1ex, ref=(\alph*)}
\crefalias{alisti}{empty}
\newlist{rlist}{enumerate}{1}
\setlist[rlist]{label=(\roman*),itemsep=.1ex, ref=(\roman*)}
\crefalias{rlisti}{empty}
\newlist{nlist}{enumerate}{1}
\setlist[nlist]{label=(\arabic*),itemsep=.1ex, ref=(\arabic*)}
\crefalias{nlisti}{empty}
\crefname{equation}{}{}
\Crefname{figure}{Figure}{Figures}
\crefname{page}{page}{pages}
\Crefname{enumi}{}{}
\Crefname{subsection}{Subsection}{Subsections}
\def\theoremname{Theorem}%
\def\propositionname{Proposition}%
\def\lemmaname{Lemma}%
\def\corollaryname{Corollary}%
\def\definitionname{Definition}%
\def\conventionname{Convention}%
\def\axiomname{Axiom}%
\def\remarkname{Remark}%
\def\examplename{Example}%
\def\questionname{Question}%
\def\constructionname{Construction}%
\def\notationname{Notation}%
\def\assumptionname{Assumption}%
\def\conjecturename{Conjecture}%
\newtheorem{thm}{\theoremname}[section]
\theoremstyle{plain}
\newtheorem{theorem}[thm]{\theoremname}
\newtheorem{proposition}[thm]{\propositionname}

\newtheorem{lemma}[thm]{\lemmaname}
\newtheorem{corollary}[thm]{\corollaryname}
\theoremstyle{definition}
\newtheorem{definition}[thm]{\definitionname}

\newtheorem{remark}[thm]{\remarkname}

\newcommand{\abs}[1]{\lvert#1\rvert}
\newcommand{\norm}[1]{\lVert#1\rVert}
\newcommand{\field}[1]{\ensuremath{\mathbb{#1}}}

\newcommand{\clopen}[2]{\char"5B #1,#2\char"29}

\DeclareMathOperator{\card}{card}

\allowdisplaybreaks
\usepackage{xcolor}

\csname enddump\endcsname
\begin{document}
	
\title[Linear equations in Lorentz spaces]{Linear functional equations and their solutions in Lorentz spaces}	

\selectlanguage{polish}
\author[J. Morawiec]{Janusz Morawiec}
\address{Instytut Matematyki{}\\ Uniwersytet Śląski{}\\	Bankowa 14, PL-40-007 Katowice{}\\ Poland}
\email{morawiec@math.us.edu.pl}
\author[T. Zürcher]{Thomas Zürcher}
\address{Instytut Matematyki{}\\ Uniwersytet Śląski{}\\ Bankowa 14, PL-40-007 Katowice{}\\ Poland}
\email{thomas.zurcher@us.edu.pl}
\selectlanguage{English}
	
\begin{abstract}
Assume that $\Omega\subset \field{R}^k$ is an open set, $V$ is a separable Banach space over a field $\mathbb K\in\{\mathbb R,\mathbb C\}$ and $f_1,\ldots,f_N \colon\Omega\to \Omega$, $g_1,\ldots, g_N\colon\Omega\to \field{K}$, $h_0\colon \Omega\to V$ are given functions. We are interested in the existence and uniqueness of solutions $\varphi\colon \Omega\to V$ of the linear functional equation $\varphi=\sum_{k=1}^{N}g_k\cdot(\varphi\circ f_k)+h_0$ in Lorentz spaces.
\end{abstract}
	
\maketitle


\section{Introduction}
Throughout this paper we fix $k,N\in\mathbb N$, an open set $\Omega\subset \field{R}^k$, a separable Banach space $(V,\|\cdot\|_V)$ over a field $\mathbb K\in\{\mathbb R,\mathbb C\}$ and functions $f_1,\ldots,f_N \colon\Omega\to \Omega$, $g_1,\ldots, g_N\colon\Omega\to \field{K}$ and $h_0\colon \Omega\to V$. We are interested in the existence and uniqueness of a special solution $\varphi\colon\Omega\to V$ of the following linear equation
\begin{equation}\label{e}
\varphi(x)=\sum_{n=1}^{N}g_n(x)\varphi(f_n(x))+h_0(x).
\end{equation}
Different solutions of equation~\cref{e} have been studied by many authors (e.g.\ \cite[Chapter~XIII]{Kuczma1968}, \cite[Chapter~6]{KuczmaChoczewskiGer1990}, \cite[Chapter~5]{BelitskiiTkachenko2003}, \cite[Section~4]{BaronJarczyk2001} and the references therein). In order to specify what special solution we are talking about, we need to introduce some notations. But before, let us note that this paper is a continuation of the authors' paper \cite{MorawiecZurcher} and a reader who is already familiar with the content of that paper may want to jump to~\cref{Introduction Lorentz}.

Denote by $\mathcal F$ the linear space of all functions $\psi\colon\Omega\to V$ and fix a subspace~$\mathcal F_0$ of~$\mathcal F$. Then define the operator $P\colon\mathcal F_0\to\mathcal F$ by
\begin{equation}\label{P}
P\psi=\sum_{n=1}^{N}g_n\cdot(\psi\circ f_n),
\end{equation}
and observe that it is linear and equation~\cref{e} can be written in the form
\begin{equation}\label{E}
\varphi=P\varphi+h_0.
\end{equation}

If equation~\cref{e} has a solution $\varphi\in\mathcal F_0$ such that $P\varphi\in\mathcal F_0$, then $h_0\in\mathcal F_0$. Conversely, if $h_0\in\mathcal F_0$, then for every solution $\varphi\in\mathcal F_0$ of equation~\cref{e} we have $P\varphi\in\mathcal F_0$. Therefore, if we want to look for solutions of equation~\cref{e} in $\mathcal F_0$, then it is quite natural to assume that $h_0\in\mathcal F_0$ and
\begin{equation}\label{PF}
P(\mathcal F_0)\subset\mathcal F_0.	
\end{equation}

\begin{remark}[{see \cite[Remark 1.1]{MorawiecZurcher}, cf.\ \cite[Remark 1.2]{MorawiecZurcher2019}}]\label{rem11}
Assume that $\mathcal F_0$ is equipped with a norm, $h_0\in\mathcal F_0$, and the operator~$P$ given by~\cref{P} satisfies~\cref{PF} and is continuous. If the series
\begin{equation}\label{elementary}
\sum_{n=0}^{\infty}P^n h_0
\end{equation}
converges, in the norm, to a function $\varphi\in\mathcal F_0$, then \cref{E} holds.
\end{remark}

From now on, the series~\cref{elementary} will be called the \emph{elementary solution} of equation~\cref{e} in $\mathcal F_0$, provided that it is a well-defined solution of  equation~\cref{e} belonging to $\mathcal F_0$. Let us note that it can happen that equation~\cref{e} has a solution in $\mathcal F_0$, however its elementary solution in $\mathcal F_0$ can fail to exist (see \cite[Example~1.4]{MorawiecZurcher2019}).

The investigation of the existence of the elementary solution of equation~\cref{e} in the case $\mathcal F_0=L^1([0,1],\field{R})$ was motivated by \cite{Nikodem1991} and studied in \cite{MorawiecZurcher2019}. Next, inspired by \cite{Matkowski1975}, the existence of the elementary solution of equation~\cref{e} in the case where $\mathcal F_0$ is a generalized Orlicz space was examined in \cite{MorawiecZurcher}.

The basic result on the existence and uniqueness of the elementary solution of equation~\cref{e} in $\mathcal F_0$ reads as follows.

\begin{theorem}[{see \cite[Theorem~1.2]{MorawiecZurcher}, cf. \cite[Theorem~3.2]{MorawiecZurcher2019}}]\label{BFS}
Assume that $\|\cdot\|$ is a complete norm in $\mathcal F_0$ and let $h_0\in\mathcal F_0$. If the operator $P$ given by~\cref{P} satisfies~\cref{PF} and is a contraction with contraction factor $\alpha$, then the elementary solution of equation~\cref{e} in $\mathcal F_0$ exists, it is the unique solution of equation~\cref{e} in $\mathcal F_0$ and $\norm{\sum_{k=m}^{\infty}P^k h_0}\leq\frac{\alpha^{m}}{1-\alpha}\norm{h_0}$.	
\end{theorem}


\section{Preliminaries}
Let $(X,\mathcal M,\mu)$  and $(Y,\mathcal N,\nu)$ be measure spaces. We say that $G\colon X\to Y$ satisfies \emph{Luzin's condition~N} if for every set $N\subset Y$ of measure zero the set $G(N)$ is also of measure zero. When we will integrate a function $\Phi\colon X\to V$, we will use the Bochner integral (for details see e.g.\ \cite[Sections 3.1~and~3.2]{HeinonenKoskelaShanmugalingamTyson2015}). Recall that a function $\Phi\colon X\to V$ is \emph{Bochner--measurable} if it is equal almost everywhere to the limit of a sequence of measurable simple functions, i.e., $\Phi(x)=\lim_{n\to\infty}\Phi_n(x)$ for almost all $x\in X$, where each of the functions $\Phi_n\colon X\to V$ has a finite range and $\Phi_n^{-1}(\{v\})$ is measurable for every $v\in V$. As we will work with Bochner--integrable solutions of equation~\cref{e}, we need the following observation.

\begin{lemma}[{see \cite[Lemma~2.1]{MorawiecZurcher}}]\label{lem21}
Assume that $(X,\mathcal M,\mu)$ is a complete $\sigma$\nobreakdash--finite measure space. Let $F\colon X\to X$, $H\colon X\to V$ and $G\colon X\to \field{K}$ be measurable functions. If for all sets $N\subset X$ of measure zero the set $F^{-1}(N)$ is also of measure zero, then
the functions $H\circ F$ and $G\cdot(H\circ F)$ are measurable.
\end{lemma}
	
The next result we want to apply is a change of variable formula from \cite{Hajlasz1993}. To formulate this theorem, we need to introduce some definitions and notions.

Let $F\colon\Omega\to\mathbb R^k$ be measurable. We say that a linear mapping $L\colon\mathbb R^k\to\mathbb R^k$ is an \emph{approximate differential} of $F$ at $x_0\in\Omega$ if for every $\varepsilon>0$ the set
\begin{equation*}
\left\{x\in\Omega\setminus\{x_0\}: \frac{\|F(x)-F(x_0)-L(x-x_0)\|}{\|x-x_0\|}<\varepsilon\right\}
\end{equation*}
has $x_0$ as a density point (see \cite[Section~2]{Whitney1951}, cf.~\cite[Chapter IX.12]{Saks1964}).
We say that $F$ is \emph{approximately differentiable} at $x_0$ if the approximate differential of $F$ at $x_0$ exists. To simplify notation, we will denote the approximate differential of a function $F\colon\Omega\to\mathbb R^k$ at~$x_0$ by $F'(x_0)$. Moreover, if a function $F\colon\Omega\to\mathbb R^k$ is almost everywhere approximately differentiable, then as usual we denote by $F'$ the function $\Omega\ni x\mapsto F'(x)$, adopting the convention that $F'(x)=0$ for every point $x\in\Omega$ at which $F$ is not approximately differentiable. If $E\subset\Omega$, then the function $N_F(\cdot,E)\colon\mathbb{R}^k\to\mathbb{N}\cup \{\infty\}$ defined by
\begin{equation*}
N_F(y,E)=\card(F^{-1}(y)\cap E)
\end{equation*}
is called the \emph{Banach indicatrix} of~$F$.

As we are working with functions equal almost everywhere, we need the following observation.

\begin{lemma}[{see \cite[Lemma 2.1]{MorawiecZurcher2019}, cf. \cite[Lemma 2.1]{MorawiecZurcher}}]\label{lemm22}
Let $F_1,F_2\colon \Omega\to\mathbb R^k$ be functions such that $F_1=F_2$ almost everywhere. If $F_1$ is approximately differentiable almost everywhere, then $F_2$ is as well. Moreover, whenever $F_1$~or~$F_2$ is approximately differentiable at a point, the other function is as well, and the approximate derivatives agree at this point.
\end{lemma}

Now we are in a position to formulate the change of variable formula in which $J_F$ denotes the determinant of the Jacobi matrix of~$F$.

\begin{theorem}[{see \cite[Theorem~2]{Hajlasz1993}}]\label{Hajlasz}
Assume that $F\colon \Omega\to\mathbb{R}^k$ is a measurable function satisfying Luzin's condition~N and being almost everywhere approximately differentiable. If $H\colon \mathbb{R}^k\to \mathbb{R}$ is a measurable function, then for every measurable set $E\subset \Omega$ the following statements are true:
\begin{enumerate}[label={\rm (\roman*)}]
\item\label{Hajlasz Measurable} The functions $(H\circ F)\abs{J_F}$ and $H N_F(\cdot,E)$ are measurable;
\item\label{Hajlasz Formula} If $H\geq 0$, then
\begin{equation}\label{change}
\int_E (H\circ F)(x)|J_F(x)|\, dx=\int_{\mathbb{R}^k}H(y)N_F(y,E)\,dy;
\end{equation}
\item\label{Hajlasz Integrability} If one of the functions $(H\circ F)\abs{J_F}$ and $H N_F(\cdot,E)$ is integrable $($for $(H\circ F)\abs{J_F}$ integrability is considered with respect to~$E$$)$, then so is the other and \cref{change} holds.
\end{enumerate}
\end{theorem}

Now we are ready to formulate the main assumption about the functions that were fixed at the beginning of this paper. The assumption reads as follows.

\begin{enumerate}[label={\rm (H)}]
\item\label{H} \textit{The functions $f_1,\ldots,f_N$ are measurable and almost everywhere approximately differentiable and satisfy Luzin's condition~N. For all $n\in\{1,\ldots, N\}$ and sets $M\subset\mathbb R^k$ of measure zero the set $f_n^{-1}(M)$ is of measure zero. There exists $K\in\mathbb N$ such that for every $n\in\{1,\ldots, N\}$ the set $\{x\in \Omega:\card{f_n^{-1}(x)}>K\}$ is of measure zero. The functions $g_1,\ldots,g_N $ and $h_0$ are measurable.
Moreover,
\begin{equation*}\label{L}
\quad\quad L=\max\{l\in\{1,\ldots,N\}:l_{n_1,\ldots,n_l}>0\hbox{ for some }n_1<n_2<\cdots<n_l\},
\end{equation*}
where $l_{n_1,\ldots,n_l}$ is the $k$-dimensional Lebesgue measure of the intersection $\bigcap_{i=1}^l f_{n_i}(\Omega)$ for all $l\in\{1,\ldots,N\}$ and distinct $n_1,\ldots,n_l\in\{1,\ldots,N\}$.}
\end{enumerate}


\section{Lorentz spaces of complex valued functions}\label{Introduction Lorentz}
In this paper we are interested in the existence of the elementary solution of equation~\cref{e} in  $\mathcal F_0$ being a Lorentz space. In fact, we are interested in assumptions guaranteeing that the elementary solution of equation~\cref{e} in a given Lorentz space exists, and moreover, that equation~\cref{e} has no other solutions in this Lorentz space.

Lorentz spaces were introduced originally in \cite{Lorentz1950}. They play an important role when studying generalizations of Sobolev maps and are used to obtain sharp conditions for Sobolev maps to be differentiable almost everywhere or to send sets of measure zero to sets of measure zero (see \cite{KauhanenKoskelaMaly1999} for more details). Lorentz spaces are also important in interpolation theory (see e.g.\ \cite[Theorem~3.5.15]{EdmundsEvans2004} or \cite[Theorem~4.4.13]{BennettSharpley1988}).

We begin our investigations with Lorentz spaces that consist of complex or real valued functions, i.e.\ with the case where $V=\mathbb K$. For the convenience of the reader, following  \cite{MalySwansonZiemer2009} we recall some basic definitions and facts for our need. More details on Lorentz spaces can be found e.g.\ in \cite{SteinWeiss1971,BennettSharpley1988,EdmundsEvans2004}.

\begin{definition}[{see \cite[Definition~2.2]{MalySwansonZiemer2009}}]\label{def31}
A non-decreasing left-continuous and convex function $\Psi\colon [0,\infty)\to[0,\infty]$ is said to be a \emph{Young function}, if
\begin{equation*}
\lim_{t\to 0+}\Psi(t)=\Psi(0)=0\quad\hbox{and}\quad\lim_{t\to \infty}\Psi(t)=\infty.
\end{equation*}
\end{definition}

From now on the symbol~$\Psi$ is reserved for Young functions only.

\begin{definition}[{see \cite[Definition~2.4]{MalySwansonZiemer2009}}]\label{def32}
A Young function $\psi\colon [0,\infty)\to[0,\infty)$ is said to satisfy \emph{condition~$\Delta_2$ globally}, if there exists a real number $d\in(1,\infty)$ such that
\begin{equation*}
\psi(2t)\leq d\psi(t)\quad \text{for every }t\in(0,\infty).
\end{equation*}
\end{definition}

To the end of this paper we fix a strictly increasing Young function $\psi$ satisfying the condition $\Delta_2$ globally and such that
\begin{equation}\label{N}
\lim_{t\to 0+}\frac{\psi(t)}{t}=\lim_{t\to \infty}\frac{t}{\psi(t)}=0.
\end{equation}

According to \cite[page~35]{MalySwansonZiemer2009} the fixed function $\psi$ has left-sided and right-sided derivatives, which coincide except on a possibly countable set. When we write $\psi_r'$, we refer henceforth to the right-sided derivative. Note that $\psi$ satisfies~\cref{N} if and only if 	$0<\psi_r'(t)<\infty$ for every $t\in(0,\infty)$, $\psi_r'(0)=0$ and 	$\lim_{t\to\infty}\psi_r'(t)=\infty$.

We now define a function $\tau\colon[0,\infty)\to[0,\infty)$ by putting
\begin{equation*}
\tau(t)=\begin{cases}
0,&\hbox{if }t=0,\\
\frac{1}{\psi(\frac{1}{t})},&\hbox{if }t\in(0,\infty).
\end{cases}
\end{equation*}
Note that the just defined function $\tau$ is a strictly increasing Young function satisfying condition $\Delta_2$ globally and
\begin{equation*}
\lim_{t\to 0+}\frac{\tau(t)}{t}=\lim_{t\to \infty}\frac{t}{\tau(t)}=0
\end{equation*}
(see \cite[page 68]{MalySwansonZiemer2009}).

Denote by $\mu$ the $k$-dimensional Lebesgue measure on $\mathbb R^k$ and by $\mathcal M_{\mathbb K}$ the space of all Lebesgue measurable functions from $\Omega$ to $\mathbb K$.

\begin{definition}[{see \cite[Definition~9.1]{MalySwansonZiemer2009}}]\label{def33}
If $f\in\mathcal M_{\mathbb K}$, then the function $\mu_f\colon \clopen{0}{\infty}\to [0,\infty]$ defined by
\begin{equation*}
\mu_f(s)=\mu(\{x\in\Omega:|f(x)|>s\})
\end{equation*}
is said to be the \emph{distribution function} of~$f$.
\end{definition}

The linear space $L^{\psi,1}(\Omega,\mathbb K)$ consisting of all functions $f\in L^1_{\textnormal{loc}}(\Omega,\mathbb K)$ (i.e.\ $f\in L^1(K,\mathbb K)$ for every compact set $K\subset\Omega$) satisfying the following condition
\begin{equation*}
\int_0^{\infty}\tau^{-1}(\mu_f(s))\, ds<\infty	
\end{equation*}
is called the \emph{Lorentz space} (see \cite[Definition~9.2]{MalySwansonZiemer2009}). The Lorentz space equipped with the norm
\begin{equation}\label{Lnorm}
\norm{f}_{L^{\psi,1}(\Omega,\mathbb K)}=\int_0^{\infty}\tau^{-1}(\mu_f(s))\, ds
\end{equation}
is a Banach space (see \cite[page 68]{MalySwansonZiemer2009}; we will give a sketch of a proof of this fact at the beginning of~\cref{LSVV}).

The next observation relates Lorentz and Orlicz spaces. But before, let us recall the definition of Orlicz spaces.

\begin{definition}[{see \cite[Definition~2.7]{MalySwansonZiemer2009}}]\label{Def Orlicz}
Let $\Omega\subset \mathbb{R}^k$ and let $\Psi\colon [0,\infty)\to[0,\infty]$ be a Young function. The \emph{Orlicz spaces} $L^{\Psi}(\Omega,\mathbb{K})$ is the set of all measurable functions~$u\colon\Omega\to\mathbb K$ such that
\begin{equation*}
\int_{\Omega}\Psi\left(\left|\frac{u(x)}{t}\right|\right)\, dx<\infty
\end{equation*}
for some $t>0$.
\end{definition}

Equipping $L^\Psi(\Omega,\mathbb{K})$ with the \emph{Luxemburg norm}
\begin{equation*}
\norm{u}_{L^{\Psi}(\Omega,\mathbb{K})}=\inf\left\{t>0: \int_{\Omega}\Psi\left(\left|\frac{u(x)}{t}\right|\right)\, dx\leq 1\right\}
\end{equation*}
it becomes a Banach space (see e.g. \cite[Theorem 2.4]{Kozek1977}).

\begin{lemma}[{see \cite[Lemmas~9.3 and 9.4]{MalySwansonZiemer2009}}]\label{Orlicz--Lorentz}
Assume that $h\in\mathcal M_{\mathbb K}$.
\begin{enumerate}[label={\rm (\roman*)}]
\item\label{LO} If $h\in L^{\psi,1}(\Omega,\mathbb K)$ with $\norm{h}_{L^{\psi,1}(\Omega,\mathbb K)}=1$, then there exists a Young function~$\Psi$ such that
\begin{equation}\label{16}
\int_{\Omega}\Psi(|h(x)|)\, dx\leq 1=\int_0^\infty(\tau_r')^{-1}\Big(\frac{1}{\psi'(t)}\Big)\, dt.
\end{equation}
\item\label{OL} If $\Psi$ is a Young function satisfying~\cref{16}, then $h\in L^{\psi,1}(\Omega,\mathbb K)$ and
    $\norm{h}_{L^{\psi,1}(\Omega,\mathbb K)}\leq 2\norm{h}_{L^{\Psi}(\Omega,\mathbb K)}$.
\end{enumerate}
\end{lemma}

The next result is a counterpart of \cite[Theorem~4.6]{MorawiecZurcher} for Lorentz spaces.

\begin{theorem}\label{Lorentz}
Assume that \cref{H} holds and let $h_0\in L^{\psi,1}(\Omega,\mathbb K)$. If there exists a real constant $\alpha\in[0,\frac{1}{2})$ such that
\begin{equation}\label{14}
\setlength{\jot}{-3pt}
\begin{split}
|g_n(x)|\leq\alpha\min\left\{\frac{|J_{f_n}(x)|}{KL},\frac{1}{N}\right\}\quad&\hbox{for all }n\in\{1,\ldots,N\}\hbox{ and}\\&\hbox{almost all }x\in \Omega,
\end{split}
\end{equation}
then the elementary solution  of equation~\cref{e} in $L^{\psi,1}(\Omega,\mathbb K)$ exists, it is the unique solution of equation~\cref{e} in $L^{\psi,1}(\Omega,\mathbb K)$ and
\begin{equation*}
\left\|\sum_{k=m}^{\infty}P^k h_0\right\|_{L^{\psi,1}(\Omega,\mathbb K)} \leq\frac{(2\alpha)^{m}}{1-2\alpha}\norm{h_0}_{L^{\psi,1}(\Omega,\mathbb K)}.
\end{equation*}	
\end{theorem}

\begin{proof}
We want to apply~\cref{BFS}. For this purpose it suffices to check that $Ph\in L^{\psi,1}(\Omega,\mathbb K)$ and $\norm{Ph}_{L^{\psi,1}(\Omega,\mathbb K)}<2\alpha$ for every  $h\in L^{\psi,1}(\Omega,\mathbb K)$ with $\norm{h}_{L^{\psi,1}(\Omega,\mathbb K)}=1$.

Fix $h\in L^{\psi,1}(\Omega,\mathbb K)$ with $\norm{h}_{L^{\psi,1}(\Omega,\mathbb K)}=1$. By assertion~\cref{LO} of~\cref{Orlicz--Lorentz} there exists a Young function $\Psi$ such that~\cref{16} holds. Thus $h\in L^\Psi(\Omega,\mathbb K)$ and $\norm{h}_{L^\Psi(\Omega,\mathbb K)}\leq 1$. Applying \cite[Remark~4.5~and~Lemma~4.4]{MorawiecZurcher} (note that we may apply these results as we are actually dealing with the norms of $h$~and~$Ph$, respectively), we conclude that $Ph\in L^\Psi(\Omega,\mathbb K)$ and $\norm{Ph}_{L^\Psi(\Omega,\mathbb K)}\leq\alpha$. This jointly with assertion~\cref{OL} of~\cref{Orlicz--Lorentz} implies that $Ph\in L^{\psi,1}(\Omega,\mathbb K)$ and
$\norm{Ph}_{L^{\psi,1}(\Omega,\mathbb K)}\leq 2\norm{Ph}_{L^{\Psi}(\Omega,\mathbb K)}\leq 2\alpha$.
\end{proof}

Now, we fix and integer $m>1$ and consider the Lorentz space $L^{\psi_m,1}(\Omega,\mathbb K)$ that is generated by the Young function of the form $\psi_m(t)=mt^m$ for every $t\in[0,\infty)$ (see \cite[Section~3.4.1]{EdmundsEvans2004}; cf.\ \cite{KauhanenKoskelaMaly1999} where mappings with derivatives in those Lorentz spaces are considered).
The following result follows from~\cref{Lorentz}.

\begin{corollary}\label{Usual Lorentz}
Assume that \cref{H} holds. Let $m>1$ and let $h\in L^{\psi_m,1}(\Omega,\mathbb K)$. If \cref{14} holds with a real constant $\alpha\in[0,\frac{1}{2})$, then the elementary solution  of equation~\cref{e} in $L^{\psi_m,1}(\Omega,\mathbb K)$ exists, it is the unique solution of equation~\cref{e} in $L^{\psi_m,1}(\Omega,\mathbb K)$ and
\begin{equation*}
\left\|\sum_{k=m}^{\infty}P^k h_0\right\|_{L^{\psi_m,1}(\Omega,\mathbb K)} \leq\frac{(2\alpha)^{m}}{1-2\alpha}\norm{h_0}_{L^{\psi_m,1}(\Omega,\mathbb K)}.
\end{equation*}	
\end{corollary}


\section{Lorentz spaces of vector valued functions}\label{LSVV}
We begin with a definition, which collects the nice properties of Lebesgue spaces. To formulate it, we denote by~$\mathcal{M}_{\mathbb{K}}^+$ the subclass of~$\mathcal{M}_{\mathbb{K}}$ of functions that are almost everywhere nonnegative and accept that the symbol $a_k\uparrow a$ means that  $(a_k)_{k\in\mathbb N}\in [0,\infty]^{\mathbb N}$ is increasing and converges to $a\in[0,\infty]$.

\begin{definition}[{see\cite[Definition~3.1.1]{EdmundsEvans2004}}]\label{BFSnorm}
A \emph{Banach function norm} on~$(\Omega,\mu)$ is a map~$\rho\colon \mathcal{M}_{\mathbb{K}}^{+}\to [0,\infty]$ such that for all functions $f,g\in \mathcal{M}_{\mathbb{K}}^+$, sequences $(f_k)_{k\in\mathbb N}$ of functions from $\mathcal{M}_{\mathbb{K}}^+$,  scalars $\lambda\geq 0$ and  $\mu$\nobreakdash-measurable sets~$E\subset \Omega$, the following conditions hold:
\begin{enumerate}[label=(P\arabic*)]
\item\label{P1} $\rho(f)=0\Longleftrightarrow f=0$ $\mu$\nobreakdash-a.e., $\rho(\lambda f)=\lambda \rho(f)$ and $\rho(f+g)\leq \rho(f)+\rho(g)$;
\item\label{P2} $g\leq f$ $\mu$\nobreakdash-a.e. $\Longrightarrow\rho(g)\leq \rho(f)$;
\item\label{P3} $f_k\uparrow f$ $\mu$\nobreakdash-a.e. $\Longrightarrow\rho(f_k)\uparrow \rho(f)$;
\item\label{P4} $\mu(E)<\infty\Longrightarrow\rho(\chi_E)<\infty$;
\item\label{P5} $\mu(E)<\infty\Longrightarrow\int_E f\, d\mu\leq C(E)\rho(f)$ with some $C(E)<\infty$.
\end{enumerate}

If $\rho$ is a Banach function norm on $(\Omega,\mu)$, then the set
\begin{equation*}
X\mathbb K=\{f\in\mathcal{M}_{\mathbb K}: \rho(|f|)<\infty\}
\end{equation*}
is called a \emph{Banach function space}; as usual, we identify functions that are equal a.e.
\end{definition}

The basic result on Banach functions spaces reads as follows.

\begin{theorem}[{see \cite[Theorem~3.1.3]{EdmundsEvans2004}}]\label{thmEE}
Any Banach function space $X\mathbb K$ equipped with the norm $\|f\|_{X\mathbb K}=\rho(|f|)$ is a Banach space.	
\end{theorem}

Now, we will give a brief sketch of the fact that the Lorentz space $L^{\psi,1}(\Omega,\mathbb K)$ equipped with the norm defined by~\cref{Lnorm} is complete. According to~\cref{thmEE} it suffices to show that the map~$\rho\colon \mathcal{M}_{\mathbb{K}}^{+}\to [0,\infty]$ given by
\begin{equation*}\label{BFN}
\rho(f)=\int_0^{\infty}\tau^{-1}(\mu_f(s))\, ds
\end{equation*}
is Banach function norm on~$(\Omega,\mu)$.

We only look at the more involved conditions of~\cref{BFSnorm}. \gdef\isBFS{} Let us focus first on the triangle inequality in~\cref{P1}. For this purpose we need the following definition in which we use the convention that $\inf \emptyset=\infty$.

\begin{definition}[{see \cite[Definition~3.2.3]{EdmundsEvans2004}}]
The \emph{non-increasing rearrangement} of an almost everywhere finite function $f\in \mathcal M_{\mathbb K}$ is the function $f^*\colon \clopen{0}{\infty}\to [0,\infty]$ defined by
\begin{equation*}
f^*(t)=\inf\{\lambda \in \clopen{0}{\infty}:\mu_f(\lambda)\leq t\}.
\end{equation*}
\end{definition}
Switching from the distribution function $\mu_f$ to the nonincreasing rearrangement $f^*$ we have
\begin{dmath}\label{strat 1}
\rho(f)=\int_{0}^{\infty}f^*(\tau(s))\, ds.
\end{dmath}
We know that $\tau$ is convex, but following the proof of the theorem in \cite{CoffeeRoom}, we can prove that it is even bi-Lipschitz on compact intervals. This allows us to use~\cref{Hajlasz} and obtain
\begin{dmath}\label{strat 2}
\rho(f)=\int_{0}^{\infty}f^*(s) (\tau^{-1})'(s)\, ds.
\end{dmath}
Since $\tau$ is convex, it has a nondecreasing right derivative (see e.g. \cite[Theorem~1.3.3]{Convex} or \cite[Theorems 3.7.3~and~3.7.4]{KK96}). Finally, replacing this derivative with a nondecreasing function that is right continuous and agree with the original function almost everywhere, we apply the following result.
\begin{proposition}[{see \cite[Proposition~2.7]{Advanced}}]\label{Subadditivity}
If $\varphi\colon (0,\infty)\to \field{R}$ is right continuous, nonnegative and nonincreasing, then the operator
\begin{equation*}
f\mapsto \int_{0}^{\infty} \varphi(t)f^{*}(t)\, dt	
\end{equation*}
is subadditive.
\end{proposition}

Condition~\cref{P3} can be proven by using the fact that~$\tau^{-1}$ is continuous on compact intervals, whereas condition~\cref{P5} can be seen by applying \cite[Proposition~3.2.5]{EdmundsEvans2004} and the fact that $t\mapsto \tau^{-1}(t)/t$ is nonincreasing, which follows from \cite[Lemma]{CoffeeRoom}.

Banach function spaces are originally defined as spaces of $\field{K}$\nobreakdash-valued functions. We will show that the corresponding spaces for functions with Banach space targets are Banach spaces as well. We begin with formal definition of these spaces.

Denote by $\mathcal M_V$ the space of all Lebesgue measurable functions from $\Omega$ to $V$.

\begin{definition}\label{BFSV}	
The set
\begin{equation*}
XV=\{f\in\mathcal M_V:\|f\|_V\in X\field{K}\}
\end{equation*}
is called a \emph{Banach function space $($of vector valued functions$)$}; as in the $\field{K}$\nobreakdash-valued case we identify functions that are equal a.e.
\end{definition}

It is easy to check that the formula
\begin{equation}\label{normV}
\norm{f}_{XV}=\norm{\norm{f}_V}_{X\field{K}}	
\end{equation}
defines a norm on the Banach function space $XV$. We want to prove that the just defined norm is complete. The next lemma, inspired by \cite[Lemma~3.1.2]{EdmundsEvans2004}, is the basic step in the proof.

\begin{lemma}\label{Banach to Banach}
Assume that $(f_n)_{n\in\mathbb N}$ ia a sequence of functions from the Banach function space $XV$ such that $\lim_{n\to \infty}f_n(x)$ exists for almost all $x\in\Omega$. If $\liminf_{n\to \infty}\norm{f_n}_{XV}<\infty$, then $\lim_{n\to \infty}f_n\in XV$ and
\begin{dmath*}
\big\|\lim_{n\hiderel{\to} \infty}f_n\big\|_{XV}\leq \liminf_{n\to \infty}\norm{f_n}_{XV}.
\end{dmath*}
\end{lemma}

\begin{proof}
It is clear that the formula
\begin{equation*}
f(x)=\begin{cases}
\lim_{n\to\infty}f_n(x),&\hbox{ if }\lim_{n\to \infty}f_n(x)\hbox{ exists}\\
0,&\mbox{otherwise}
\end{cases}
\end{equation*}
defines a function belonging to $\mathcal M_V$.

For every $n\in\mathbb N$ we define the function $g_n\colon \Omega\to[0,\infty]$ by putting
\begin{equation*}
g_n(x)=\inf_{m\geq n}\norm{f_m(x)}_V.
\end{equation*}
Then for almost all $x\in\Omega$ we have
\begin{dmath*}
\lim_{n\hiderel{\to}\infty}g_n(x)=\lim_{n\to \infty}\inf_{m\geq n}\norm{f_m(x)}_V =\liminf_{n\to \infty}\norm{f_n(x)}_V=\norm{f(x)}_V.
\end{dmath*}
Hence
\begin{dmath}\label{inf trick}
\norm{f}_{XV}=\norm{\norm{f}_V}_{X\field{K}}=\norm{\lim_{n\to \infty}g_n}_{X\field{K}}=\lim_{n\to \infty}\norm{\inf_{m\geq n}\norm{f_m}_V}_{X\field{K}}.
\end{dmath}

Since for all $l,n\in\mathbb N$ with $l\geq n$ we have $\inf_{m\geq n} \norm{f_m}_V\leq \norm{f_l}_V$, it follows by~\cref{P2} that
\begin{dmath*}
\norm{\inf_{m\geq n} \norm{f_m}_V}_{X\field{K}}\leq \norm{\norm{f_l}_V}_{X\field{K}}.
\end{dmath*}
Therefore,
\begin{dmath*}
\norm{\inf_{m\geq n} \norm{f_m}_V}_{X\field{K}}=\inf_{l\geq n} \norm{\inf_{m\geq n} \norm{f_m}_V}_{X\field{K}}\leq \inf_{l\geq n}\norm{\norm{f_l}_V}_{X\field{K}},
\end{dmath*}
which jointly with~\cref{inf trick} gives
\begin{dmath*}
\norm{f}_{XV}\leq \lim_{n\to \infty}\inf_{l\geq n}\norm{\norm{f_l}_V}_{X\field{K}}
=\liminf_{n\to \infty}\norm{\norm{f_n}_V}_{X\field{K}}
=\liminf_{n\to \infty}\norm{f_n}_{XV}<\infty.
\end{dmath*}
In consequence, $f\in XV$.
\end{proof}

\begin{theorem}\label{is BS}
The Banach function space $XV$ equipped with the norm defined by \eqref{normV} is a Banach space.	
\end{theorem}

\begin{proof}
We will follow the proof of \cite[Theorem~3.1.3]{EdmundsEvans2004}.

Assume that $(f_n)_{n\in\mathbb N}$ is a Cauchy sequence of functions from~$XV$ and let $(g_n)_{n\in\mathbb N}$ be one of its subsequences such that
$\|g_{n+1}-g_n\|_{XV}\leq\frac{1}{2^{n+1}}$ for every $n\in\mathbb N$.
Setting $g_0=0$, for every $n\in \field{N}$ we put $h_n=g_n-g_{n-1}$ and note that $h_n\in\mathcal M_V$. Next for all $x\in\Omega$ and $N\in\mathbb N$ we put
\begin{equation*}
G_N(x)=\sum_{n=1}^{N}\norm{h_n(x)}_V\quad\hbox{and}\quad
G(x)=\sum_{n=1}^{\infty}\norm{h_n(x)}_V.
\end{equation*}
Then $G_N,G\in\mathcal M_{\mathbb K}$ and
\begin{dmath}\label{GN finite}
\norm{G_N}_{X\field{K}}\leq\sum_{n=1}^{\infty}\norm{\norm{h_n}_V}_{X\field{K}}
=\sum_{n=1}^{\infty}\norm{h_n}_{XV}\leq \norm{h_1}_{XV}+1<\infty
\end{dmath}
for every $N\in\mathbb N$. Since $G_N\uparrow G$, it follows by~\cref{P3} that $G\in X\field{K}$.

Fix $\varepsilon>0$ and $E\subset \Omega$ such that $\mu(E)<\infty$. By~\cref{P5}, we see that
\begin{align*}\label{meas conv}
\lim_{N\to\infty}\mu(\{x\in E: \abs{G(x)-G_N(x)}>\varepsilon\})&\leq\lim_{N\to\infty} \frac{1}{\varepsilon}\int_{E}\abs{G(x)-G_N(x)}\, d\mu(x)\\
&\leq\frac{C(E)}{\varepsilon}\lim_{N\to\infty}\norm{\abs{G-G_N}}_{X\field{K}}\\
&=\frac{C(E)}{\varepsilon}\lim_{N\to\infty}\sum_{n=N}^{\infty} \norm{g_{n+1}-g_n}_{XV}\\
&=\frac{C(E)}{\varepsilon}\lim_{N\to\infty}\frac{1}{2^N}=0.
\end{align*}
Thus the sequence $(G_N)_{N\in\mathbb N}$ converges to~$G$ in Lebesgue measure on~$E$. Applying the Riesz theorem (see e.g.\ \cite[Theorem~11.26]{hewitt2012real}) there exists a subsequence of~$(G_N)_{N\in\mathbb N}$ which converges to~$G$ a.e.\ on~$E$. Since the $k$\nobreakdash-dimensional Lebesgue measure is $\sigma$\nobreakdash-finite, we can apply the denationalization method to obtain that there exists a subsequence of $(G_N)_{N\in\mathbb N}$ which converges to $G$ a.e.\ on $\Omega$. As we know that  $G\in X\field{K}$, it follows that the series $\sum_{n=1}^{\infty}\norm{h_n(x)}_V$ is finite for almost all $x\in\Omega$. Thus
\begin{dmath*}
\sum_{n\hiderel{=}1}^{\infty}h_n(x)\in V
\end{dmath*}
for almost all $x\in \Omega$.

Now, we define the function $g\in XV$ by putting
\begin{equation*}
g(x)=\begin{cases}
\lim_{n\to\infty}g_n(x),&\hbox{ if }\sum_{n\hiderel{=}1}^{\infty}h_n(x)\in V\\
0,&\mbox{otherwise}.
\end{cases}
\end{equation*}
Our goal is to show that $g$ is the limit of~$(f_n)_{n\in\mathbb N}$; note that to achieve this, it is enough to show that it is the limit of~$(g_n)_{n\in\mathbb N}$.

Fix $m\in \field{N}$. Then
\begin{equation}\label{liminf}
\begin{split}
\liminf_{n\hiderel{\to} \infty}\norm{g_m-g_n}_{XV}&=\liminf_{n\to \infty}\Big\|\sum_{k=m+1}^{n}h_k\Big\|_{XV}\leq \liminf_{n\to \infty}\sum_{k=m+1}^{\infty}\norm{h_k}_{XV}\\
&\leq\frac{1}{2^m}.
\end{split}	
\end{equation}
Applying now~\cref{Banach to Banach} we conclude that $g_m-g\in XV$ and $\norm{g_m-g}_{XV}\leq \liminf_{n\to \infty}\norm{g_m-g_n}_{XV}$. Hence  $g=g_m-(g_m-g)\in XV$ and by~\cref{liminf} we have
\begin{dmath*}
\lim_{m\hiderel{\to} \infty}\norm{g_m-g}_{XV}\leq \lim_{m\to \infty}\liminf_{n\to \infty}\norm{g_m-g_n}_{XV}
\leq \lim_{m\to \infty} \frac{1}{2^m}=0.
\end{dmath*}
Hence $(g_n)_{n\in\mathbb N}$ converges to some~$g$. Since $(f_n)_{n\in\mathbb N}$ is a Cauchy sequence and $(g_n)_{n\in\mathbb N}$ one of its subsequences, it follows that $(f_n)_{n\in\mathbb N}$ converges to $g$, which completes the proof.
\end{proof}

We end this paper with a counterpart of~\cref{Lorentz} for vector valued functions. We omit its proof as the boundedness of the considered operator can be proven by looking at the norm of the function instead of at the function itself.

\begin{theorem}\label{LorentzV}
Assume that \cref{H} holds and let $h_0\in L^{\psi,1}(\Omega,V)$. If there exists a real constant $\alpha\in[0,\frac{1}{2})$ such that
\begin{equation}\label{l14V}
\setlength{\jot}{-3pt}
\begin{split}
|g_n(x)|\leq\alpha\min\left\{\frac{|J_{f_n}(x)|}{KL},\frac{1}{N}\right\}\quad&\hbox{for all }n\in\{1,\ldots,N\}\hbox{ and}\\&\hbox{almost all }x\in \Omega,
\end{split}
\end{equation}
then the elementary solution  of equation~\cref{e} in $L^{\psi,1}(\Omega,V)$ exists, it is the unique solution of equation~\cref{e} in $L^{\psi,1}(\Omega,V)$ and
\begin{equation*}
\left\|\sum_{k=m}^{\infty}P^k h_0\right\|_{L^{\psi,1}(\Omega,V)} \leq\frac{(2\alpha)^{m}}{1-2\alpha}\norm{h_0}_{L^{\psi,1}(\Omega,V)}.
\end{equation*}	
\end{theorem}


{\bf Acknowledgement.} The research was supported by the University of Silesia  Mathematics Department (Iterative Functional Equations and Real Analysis program).

\bibliographystyle{plain}
\bibliography{bibliography}
\end{document}